\documentclass[10pt]{article}
\usepackage{geometry, enumerate, amsthm, amsmath, amssymb, amscd, verbatim, color, graphicx}
\usepackage{hyperref}
\usepackage{courier}
\usepackage[all]{xy}

\numberwithin{equation}{section}

\newtheorem{Prop}[equation]{Proposition}
\newtheorem{Thm}[equation]{Theorem}
\newtheorem{Lem}[equation]{Lemma}
\newtheorem{Cor}[equation]{Corollary}
\theoremstyle{definition}\newtheorem{Def}[equation]{Definition}
\newtheorem{Ex}[equation]{Example}
\newtheorem{Rem}[equation]{Remark}
\theoremstyle{definition}

\newcommand{\N}{\mathbb{N}}
\newcommand{\Q}{\mathbb{Q}}
\newcommand{\Z}{\mathbb{Z}}
\newcommand{\F}{\mathbb{F}}
\newcommand{\Int}{\textnormal{Int}}
\newcommand{\nil}{\textnormal{nil}}

\newcommand{\mfA}{\mathfrak{A}}
\newcommand{\mfP}{\mathfrak{P}}
\newcommand{\mfQ}{\mathfrak{Q}}
\newcommand{\mfa}{\mathfrak{a}}
\newcommand{\olA}{\overline{A}}

\newcommand{\whA}{\widehat{A}}
\newcommand{\whD}{\widehat{D}}

\title{Non-triviality Conditions for\\ Integer-valued Polynomial Rings on Algebras}

\date{\today}

\author{
G.\ Peruginelli\footnote{Via Pietro Coccoluto Ferrigni 68, 57125 Livorno, Italy. E-mail: g.peruginelli@tiscali.it}\\
N. J. Werner\footnote{Department of Mathematics, Computer and Information Science, State University of New York College at Old Westbury, Old Westbury, NY. E-mail: wernern@oldwestbury.edu}
}

\begin{document}


\maketitle

\begin{abstract}
\noindent Let $D$ be a commutative domain with field of fractions $K$ and let $A$ be a torsion-free $D$-algebra such that $A \cap K = D$. The ring of integer-valued polynomials on $A$ with coefficients in $K$ is $\Int_K(A) = \{f \in K[X] \mid f(A) \subseteq A\}$, which generalizes the classic ring $\Int(D) = \{f \in K[X] \mid f(D) \subseteq D\}$ of integer-valued polynomials on $D$.

The condition on $A \cap K$ implies that $D[X] \subseteq \Int_K(A) \subseteq \Int(D)$, and we say that $\Int_K(A)$ is nontrivial if $\Int_K(A) \ne D[X]$. For any integral domain $D$, we prove that if $A$ is finitely generated as a $D$-module, then $\Int_K(A)$ is nontrivial if and only if $\Int(D)$ is nontrivial. When $A$ is not necessarily finitely generated but $D$ is Dedekind, we provide necessary and sufficient conditions for $\Int_K(A)$ to be nontrivial. These conditions also allow us to prove that, for $D$ Dedekind, the domain $\Int_K(A)$ has Krull dimension 2.\\


\noindent MSC Primary 13F20 Secondary 13B25, 11C99
\end{abstract}

\section{Introduction}\label{Intro}
Given a (commutative) integral domain $D$ with fraction field $K$, we define $\Int(D) := \{f \in K[X] \mid f(D) \subseteq D\}$, which is the ring of integer-valued polynomials on $D$. Integer-valued polynomials and the properties of $\Int(D)$ have been well studied; the book \cite{CaCh} covers the major theory in this area and provides an extensive bibliography. In recent years, researchers have begun to study a generalization of $\Int(D)$ to polynomials that act on a $D$-algebra rather than on $D$ itself \cite{EFJ}, \cite{EvrJoh}, \cite{Fri1}, \cite{Fri1Corr}, \cite{Fri2}, \cite{LopWer}, \cite{Per1}, \cite{PerDivDiff}, \cite{PerFinite}, \cite{PerWer}, \cite{PerWer1}, \cite{Wer}. For this generalization, we let $A$ be a torsion-free $D$-algebra such that $A \cap K = D$, and let $B = K \otimes_D A$, which is the extension of $A$ to a $K$-algebra. By identifying $K$ and $A$ with their images under the injections $k \mapsto k \otimes 1$ and $a \mapsto 1 \otimes a$, we can evaluate polynomials in $K[X]$ at elements of $A$. This allows us to define $\Int_K(A) := \{f \in K[X] \mid f(A) \subseteq A\}$, which is the ring of integer-valued polynomials on $A$ with coefficients in $K$. With notation as above, the condition $A \cap K = D$ ensures that $D[X] \subseteq \Int_K(A) \subseteq \Int(D)$.

\begin{Def}\label{Nontrivial}
We say that $\Int_K(A)$ is \textit{nontrivial} if $\Int_K(A) \ne D[X]$.
\end{Def}

The goal of this paper is to determine when $\Int_K(A)$ is nontrivial. Some results in this direction were proved by Frisch in \cite[Lem. 4.1]{Fri2} and \cite[Thm. 4.3]{Fri2}; these are restated below in Proposition \ref{Fin gen nontrivial}. In the traditional case, necessary and sufficient conditions for $\Int(D)$ to be nontrivial were given by Rush in \cite{Rush}. Using Rush's criteria, we prove (Theorem \ref{Nontriv fin gen criterion}) that when $D$ is any integral domain and $A$ is finitely generated as a $D$-module, $\Int_K(A)$ is nontrivial if and only if $\Int(D)$ is nontrivial. Part of this work involves conditions under which we have $D[X] \subseteq \Int_K(M_n(D)) \subseteq \Int_K(A)$ for some $n$, where $M_n(D)$ is the algebra of $n \times n$ matrices with entries in $D$. This led us to investigate whether having $\Int_K(M_n(D)) = \Int_K(A)$ implies that $A \cong M_n(D)$. While this is not true in general, the result does hold if $D$ is a Dedekind domain and $A$ can be embedded in $M_n(D)$ (Theorem \ref{Uniqueness of M_n(D)}).

If we drop the assumption that $A$ is finitely generated as a $D$-module, determining whether $\Int_K(A)$ is nontrivial becomes more complicated. However, when $D$ is Dedekind, we are able to give necessary and sufficient conditions for $\Int_K(A)$ to be nontrivial (Theorem \ref{criterion}). Our work on this topic also allows us to prove that if $D$ is Dedekind, then $\Int_K(A)$ has Krull dimension 2 (Corollary \ref{Krull dimension}). This generalizes another theorem of Frisch \cite[Thm. 5.4]{Fri1} where it was assumed that $A$ was finitely generated as a $D$-module.

\section{Integral Algebras of Bounded Degree}\label{Non-trivial}
Throughout, $D$ denotes an integral domain with field of fractions $K$, and $A$ denotes a $D$-algebra. We will always assume that $A$ satisfies certain conditions, which we call our \textit{standard assumptions}.

\begin{Def}\label{Standard assumptions}
When $A$ is a torsion-free $D$-algebra such that $A \cap K = D$, we say that $A$ is a $D$-algebra with \textit{standard assumptions}. When $A$ is finitely generated as a $D$-module, we say that $A$ is of \textit{finite type}.
\end{Def}

As mentioned in the introduction, the condition that $A \cap K = D$ implies that
\begin{equation*}
D[X] \subseteq \Int_K(A) \subseteq \Int(D)
\end{equation*}
and it is natural to consider when $D[X] = \Int_K(A)$ or $\Int_K(A) = \Int(D)$. This latter equality is investigated in \cite{IntdecompII}, where the following theorem is proved. Unless stated otherwise, all isomorphisms are ring isomorphisms.

\begin{Thm}\label{Int_K(A)=Int(D)}
\cite[Thms. 2.10, 3.10]{IntdecompII} Let $D$ be a Dedekind domain with finite residue rings. Let $A$ be a $D$-algebra of finite type with standard assumptions. For each maximal ideal $P$ of $D$, let $\whA_P$ and $\whD_P$ be the $P$-adic completions of $A$ and $D$, respectively. Then, the following are equivalent.
\begin{enumerate}[(1)]
\item $\Int_K(A) = \Int(D)$.
\item For each nonzero prime $P$ of $D$, there exists $t \in \N$ such that $A/PA \cong \bigoplus_{i=1}^t D/P$.
\item For each nonzero prime $P$ of $D$, there exists $t \in \N$ such that $\whA_P \cong \bigoplus_{i=1}^t \whD_P$.
\end{enumerate}
\end{Thm}

In this paper, we examine the containment $D[X] \subseteq \Int_K(A)$. In the traditional setting of integer-valued polynomials, the ring $\Int(D)$ is said to be \textit{trivial} if $\Int(D) = D[X]$, and we adopt the same terminology for $\Int_K(A)$. Clearly, for $\Int_K(A)$ to be nontrivial it is necessary that $\Int(D)$ be nontrivial, so we begin by reviewing the situation for $\Int(D)$. Section I.3 of \cite{CaCh} and a paper by Rush \cite{Rush} give several results regarding the triviality or non-triviality of $\Int(D)$. We will summarize these theorems after recalling several definitions.

\begin{Def}
An ideal $\mfa$ of $D$ is said to be the colon ideal or conductor ideal of $q \in K$ if 
$$\mfa = (D :_D q) = \{d \in D \mid dq \in D\}.$$
For a commutative ring $R$, we denote by $\nil(R)$ the nilradical of $R$, which is the set of all nilpotent elements of $R$, or, equivalently, the intersection of all nonzero prime ideals of $R$. For $x \in \nil(R)$, we let $\nu(x)$ equal the nilpotency of $x$, i.e., the smallest positive integer $n$ such that $x^n = 0$. If $I\subseteq R$ is an ideal, let $V(I) = \{P \in \text{Spec}(R) \mid P \supseteq I\}$.
\end{Def}

The following proposition summarizes several sufficient and necessary conditions on $D$ in order for $\Int(D)$ to be nontrivial.
\begin{Prop}\label{IntD nontrivial}\mbox{}
\begin{enumerate}[(1)]
\item \cite[Cor. I.3.7]{CaCh} If $D$ is a domain with all residue fields infinite, then $\Int(D)$ is trivial.
\item \cite[Prop. I.3.10]{CaCh} Let $D$ be a domain. If there is a proper conductor ideal $\mfa$ of $D$ such that $D/\mfa$ is finite, then $\Int(D)$ is nontrivial.
\item \cite[Thm. I.3.14]{CaCh} Let $D$ be a Noetherian domain. Then, $\Int(D)$ is nontrivial if and only if there is a prime conductor ideal of $D$ with finite residue field.
\item \cite[Cor. 1.7]{Rush} Let $D$ be an integral domain. Then, the following are equivalent:
\begin{enumerate}[(i)]
\item $\Int(D)$ is nontrivial.
\item There exist $a,b\in D$ with $b\notin aD$ such that the two sets $\{\;|D/P|\; \mid P \in V((aD:b))\}$ and $\{\nu(x) \mid x\in {\rm nil}(D/(aD:b))\}$ are bounded.
\end{enumerate}
\end{enumerate}
\end{Prop}

If $A$ is finitely generated as a $D$-module, Frisch has shown that the analogs of the above conditions in Proposition \ref{IntD nontrivial} hold for $\Int_K(A)$:

\begin{Prop}\label{Fin gen nontrivial}
Let $D$ be a domain. Let $A$ be a $D$-algebra of finite type with standard assumptions.
\begin{enumerate}[(1)]
\item \cite[Lem. 4.1]{Fri2} Assume there is a proper conductor ideal $\mfa$ of $D$ such that $D/\mfa$ is finite. Then, $\Int_K(A)$ is nontrivial.
\item \cite[Thm. 4.3]{Fri2} Assume that $D$ is Noetherian. Then, $\Int_K(A)$ is nontrivial if and only if there is a prime conductor ideal of $D$ with finite residue field.
\end{enumerate}
\end{Prop}

In particular, \cite[Thm. 4.3]{Fri2} shows that for a Noetherian domain $D$ and a finitely generated algebra $A$, $\Int_K(A)$ is nontrivial if and only if $\Int(D)$ is nontrivial. In Theorem \ref{Nontriv fin gen criterion}, we will show that this holds even if $D$ is not Noetherian. Additionally, we can weaken our assumptions on $A$. Recall the following definition, which can be found in \cite{Jac} or \cite{Lam}, among other sources.

\begin{Def}
Let $R$ be a commutative ring and $A$ an $R$-algebra. We say that $A$ is an \emph{algebraic algebra} (over $R$) if every element of $A$ satisfies a polynomial equation with coefficients in $R$. We say that $A$ is an \emph{algebraic algebra of bounded degree} if there exists $n\in\N$ such that the degree of the minimal polynomial equation of each of its elements is bounded by $n$. If we insist that each element of $A$ satisfy a monic polynomial with coefficients in $R$, then we say that $A$ is an \emph{integral algebra} over $R$.
\end{Def}

Algebraic algebras are usually discussed over fields, in which case an algebraic algebra is also an integral algebra. Over a domain however, the two structures are not equivalent. For example, $A = \Z[\frac{1}{2}]$ is an algebraic algebra over $\Z$ that is not an integral algebra. In this case, $A$ does not satisfy our standard assumption that $A \cap \Q$ should equal $\Z$. However, if we instead take $A = \Z \oplus \Z[\frac{1}{2}]$ (so that $B = \Q \otimes_\Z A \cong \Q \oplus \Q$, $D$ is the diagonal copy of $\Z$ in $B$, and $K$ is the diagonal copy of $\Q$ in $B$), then $A$ is an algebraic algebra over $D$, $A$ is not an integral algebra over $D$, and $A \cap K = D$.

Note also that if $A$ is finitely generated as a $D$-module, then $A$ is an integral algebra of bounded degree, with the bound given by the number of generators (see \cite[Thm. 1, Chap. V]{BourbakiAlg} or \cite[Prop. 2.4]{AtMc}). However, the converse does not hold. For instance, $A = D[X_1, X_2, \ldots]/(\{X_i X_j \mid i, j \geq 1 \})$ is not finitely generated, but if $f \in A$ with constant term $d \in D$, then $f$ satisfies the polynomial $(X-d)^2$. Thus, this $A$ is an integral algebra of bounded degree. 

For our purposes, the importance of having a bounding degree $n$, is that it guarantees that $\Int_K(A)$ contains $\Int_K(M_n(D))$, where $M_n(D)$ denotes the algebra of $n \times n$ matrices with entries in $D$.

\begin{Lem}\label{Matrix containment}
Let $D$ be a domain and let $A$ be a $D$-algebra with standard assumptions. Assume that $A$ is an integral $D$-algebra of bounded degree $n$. Then, $\Int_K(M_n(D)) \subseteq \Int_K(A)$.
\end{Lem}
\begin{proof}
Let $a \in A$ and let $\mu_a \in D[X]$ be monic of degree $n$ such that $\mu_a(a) = 0$. Let $f(x) = g(X)/d \in \Int_K(M_n(D))$, where $g \in D[X]$ and $d \in D \setminus \{0\}$. By \cite[Lem. 3.4]{FriSep}, $g$ is divisible modulo $dD[X]$ by every monic polynomial in $D[X]$ of degree $n$; hence, $\mu_a$ divides $g$ modulo $d$. It follows that $g(a) \in dA$ and $f(a) \in A$. Since $a$ was arbitrary, $f \in \Int_K(A)$.
\end{proof}

\begin{Rem}
The converse of Lemma \ref{Matrix containment} does not hold, even in the case when $\Int_K(M_n(D))$ is nontrivial, as Example \ref{integralalgebraboundeddegree not necessary} below will show.
\end{Rem}

Thus, in the case of an integral algebra of bounded degree $n$, to prove that $\Int_K(A)$ is nontrivial it suffices to show that $\Int_K(M_n(D))$ is nontrivial. This task is more tractable, because the polynomials given in the next definition can be used to map $M_n(D)$ into $M_n(P)$, where $P$ is a maximal ideal of $D$ with a finite residue field.

\begin{Def}\label{BCL polynomials}
For each prime power $q$ and each $n > 0$, let 
\begin{equation*}
\phi_{q,n}(X) = (X^{q^n} - X)(X^{q^{n-1}}-X) \cdots (X^q - X).
\end{equation*}
\end{Def}

\begin{Lem}\label{BCL lemma}
\cite[Thm. 3]{BrawCarLev} Let $\F_q$ be the finite field with $q$ elements. Then, $\phi_{q,n}$ sends each matrix in $M_n(\F_q)$ to the zero matrix. Consequently, if $P \subset D$ is a maximal ideal of $D$ with residue field $D/P \cong \F_q$, then $\phi_{q, n}$ maps $M_n(D)$ into $M_n(P)$.
\end{Lem}

\begin{Prop}\label{Nontriv matrix criterion}
Let $D$ be a domain. If $\Int(D)$ is nontrivial, then $\Int_K(M_n(D))$ is nontrivial, for all $n \geq 1$.
\end{Prop}
\begin{proof}
Let $n \geq 1$ be fixed. Since $\Int(D)$ is nontrivial, by \cite[Cor. 1.7]{Rush} there exist $a,b\in D$ with $b\notin aD$ such that $\{\;|D/P|\; \mid P \in V((aD:b))\}$ and $\{\nu(x) \mid x\in \nil(D/(aD:b))\}$ are bounded. Let $I = (aD:b)$. Note that, because the former condition holds, each prime ideal containing $I$ is maximal, so the nilradical of $D/I$ is equal to the Jacobson radical of $D/I$.

Let $\{q_1,\ldots,q_s\}=\{\;|D/P|\; \mid P\in V(I)\}$. By Lemma \ref{BCL lemma}, we have $\phi_{q,n}(M_n(D))\subseteq M_n(P)$ for each maximal ideal $P\subset D$  whose residue field has cardinality $q$. Then
\begin{equation*}
g(X)=\prod_{i=1,\ldots,s}\phi_{q_i,n}(X)
\end{equation*}
is a monic polynomial such that $g(M_n(D))\subseteq \prod_{i}M_n(P_i)\subseteq M_n(J)$, where $J=\sqrt{I}$. Considering everything modulo $I$, we have $\overline{g}(M_n(D/I))\subseteq M_n(J/I)$.  

Now, since $\{\nu(x) \mid x\in \nil(D/I)\}$ is bounded, the nilpotency of every element in $J/I$ is bounded by some positive integer $t$. It is a standard exercise that a matrix over a commutative ring with nilpotent entries is a nilpotent matrix. Moreover, it easily follows that the nilpotency of every matrix in $M_n(J/I)$ is bounded by some $m\in\N$, depending only on $t$ and $n$. Hence, $\overline{g}(X)^{m}$ maps every matrix $M_n(D/I)$ to $0$, so that $g(X)^{m}$ maps $M_n(D)$ into $M_n(I)$. Finally, it is now easy to see that the polynomial $\frac{b}{a}\cdot g(X)^m$ is in $\Int_K(M_n(D))$ but not in $D[X]$.
\end{proof}

Combining Lemma \ref{Matrix containment} with Proposition \ref{Nontriv matrix criterion}, we obtain our desired theorem.

\begin{Thm}\label{Nontriv fin gen criterion}
Let $D$ be a domain and let $A$ be $D$-algebra with standard assumptions. Assume that $A$ is an integral $D$-algebra of bounded degree. Then, $\Int_K(A)$ is nontrivial if and only if $\Int(D)$ is nontrivial. In particular, if $A$ is finitely generated as a $D$-module, then $\Int_K(A)$ is nontrivial if and only if $\Int(D)$ is nontrivial.
\end{Thm}

Lemma \ref{Matrix containment} shows that, for an integral algebra $A$ of bounded degree $n$, the following containments hold:
\begin{equation*}
D[X] \subseteq \Int_K(M_n(D)) \subseteq \Int_K(A) \subseteq \Int(D).
\end{equation*}
While our focus has been on whether $\Int_K(A)$ equals $D[X]$, for the remainder of this section we will consider the containment $\Int_K(M_n(D)) \subseteq \Int_K(A)$. In particular, we will examine to what extent $\Int_K(M_n(D))$ is unique among rings of integer-valued polynomials. That is, if $\Int_K(M_n(D)) = \Int_K(A)$, then can we conclude that $A \cong M_n(D)$? In general, the answer is no, as we show below in Example \ref{Quaternion example}. However, in Theorem \ref{Uniqueness of M_n(D)} we will prove that for $D$ Dedekind, if $A$ can be embedded in $M_n(D)$, then having $\Int_K(M_n(D)) = \Int_K(A)$ implies that $A \cong M_n(D)$.

We first recall the definition of a null ideal of an algebra.

\begin{Def}\label{Null ideal}
Let $R$ be a commutative ring and $A$ an $R$-algebra. The \textit{null ideal} of $A$ with respect to $R$, denoted $N_R(A)$, is the set of polynomials in $R[X]$ that kill $A$. That is, $N_R(A) = \{ f \in R[X] \mid f(A) = 0\}$. In particular, $N_{D/P}(A/PA) = \{f \in (D/P)[X] \mid f(A/PA) = 0\}$ denotes the null ideal of $A/PA$ with respect to $D/P$.
\end{Def}

There is a close relationship between polynomials in rings of integer-valued polynomials and polynomials in null ideals.

\begin{Lem}\label{Null ideal lemma}
Let $D$ be a domain and let $A$ and $A'$ be $D$-algebras with standard assumptions.
\begin{enumerate}[(1)]
\item Let $g(X)/d \in K[X]$, where $g \in D[X]$ and $d \ne 0$. Then, $g(X)/d \in \Int_K(A)$ if and only if the residue of $g$ (mod $d$) is in $N_{D/dD}(A/dA)$. 
\item $\Int_K(A) = \Int_K(A')$ if and only if $N_{D/dD}(A/dA) = N_{D/dD}(A'/dA')$ for all $d \in D$.
\end{enumerate}
\end{Lem}
\begin{proof}
Notice that $g \in \Int_K(A)$ if and only if $g(A) \subseteq dA$ if and only if $g(A/dA) = 0$ mod $d$. This proves (1), and (2) follows easily.
\end{proof}

\begin{Ex}\label{Quaternion example}
Let $D = \Z_{(p)}$ be the localization of $\Z$ at an odd prime $p$. Take $A$ to be the quaternion algebra $A = D \oplus D\mathbf{i} \oplus D\mathbf{j} \oplus D\mathbf{k}$, where $\mathbf{i}$, $\mathbf{j}$, and $\mathbf{k}$ are the imaginary  quaternion units satisfying $\mathbf{i}^2 = \mathbf{j}^2 = -1$ and $\mathbf{i}\mathbf{j} = \mathbf{k} = -\mathbf{j}\mathbf{i}$. It is well known (cf.\ \cite[Exercise 3A]{Goodearl} or \cite[Sec. 2.5]{DavSarVal}) that $A/p^k A \cong M_2(\Z/p^k \Z) \cong M_2(D/p^k D)$ for all $k > 0$. By Lemma \ref{Null ideal lemma}, $\Int_\Q(A) = \Int_\Q(M_2(D))$. However, $A$ contains no nonzero nilpotent elements (and is in fact contained in the division ring $\Q \oplus \Q\mathbf{i} \oplus \Q\mathbf{j} \oplus \Q\mathbf{k}$) and so cannot be isomorphic to $M_2(D)$.
\end{Ex}

Thus, in general, $\Int_K(A) = \Int_K(M_n(D))$ does not imply that $A \cong M_n(D)$. However, as mentioned above, we do have such an isomorphism if $A$ can be embedded in $M_n(D)$. Proving this theorem involves some results of Racine \cite{Racine}, \cite{Racine2} about maximal subalgebras of matrix rings, which we now summarize.

\begin{Prop}\label{Racine classification}\mbox{}
\begin{enumerate}[(1)]
\item (\cite[Thm. 1]{Racine}) Let $\olA$ be a maximal $\F_q$-subalgebra of $M_n(\F_q)$. Let $V$ be an $\F_q$-vector space of dimension $n$, so that $M_n(\F_q)\cong{\rm End}_{\F_q}(V)$. Then, $\olA$ is one of the following two types.
\begin{itemize}
\item[(I)] The stabilizer of a proper nonzero subspace of $V$. That is, $\olA = S(W) = \{\varphi\in {\rm End}_{\F_q}(V) \mid \varphi(W)\subseteq W\}$, where $W$ is a proper nonzero $\F_q$-subspace of $V$.

\item[(II)] The centralizer of a minimal field extension of $\F_q$. That is, $\olA = C_{{\rm End}_{\F_q}(V)}(\F_{q^l})=\{\varphi\in{\rm End}_{\F_q}(V) \mid \varphi x=x\varphi, \forall x\in \F_{q^l} \}$, where $l\in\Z$ is a prime dividing $n$.
\end{itemize}

\item (\cite[Theorem p.\ 12]{Racine2}) Let $D$ be a Dedekind domain and let $A$ be a maximal $D$-subalgebra of $M_n(D)$. Then, there exists a maximal ideal $P$ of $D$ such that $A/P A$ is a maximal subalgebra of $M_n(D/P)$. 
\end{enumerate}
\end{Prop}

Racine's classification allows us to establish a partial uniqueness result for the null ideal of $M_n(\F_q)$, and hence for $\Int_K(M_n(D))$.

\begin{Lem}\label{FqsubalgebrasMnFq}
Let $\olA$ be an $\F_q$-subalgebra of $M_n(\F_q)$ such that $N_{\F_q}(\olA)=N_{\F_q}(M_n(\F_q))$. Then $\olA = M_n(\F_q)$.
\end{Lem}
\begin{proof}
Suppose the claim is not true, so that $\overline A$ is contained in a maximal $\F_q$-subalgebra of $M_n(\F_q)$; hence, without loss of generality, we may assume that $\overline A\subsetneq M_n(\F_q)$ is a maximal $\F_q$-subalgebra. We will show that $N_{\F_q}(\overline{A})$ properly contains $N_{\F_q}(M_n(\F_q))$. Note that $N_{\F_q}(M_n(\F_q)) = (\phi_{q,n}(X))$ by \cite[Thm. 3]{BrawCarLev}, where $\phi_{q,n}$ is the polynomial from Definition \ref{BCL polynomials}.

Let $V$ be an $\F_q$-vector space of dimension $n$, so that $M_n(\F_q)\cong{\rm End}_{\F_q}(V)$. Assume first that $\olA = S(W)$ is of Type I as in Proposition \ref{Racine classification}, and let $m = \dim_{\F_q}(W)$. Note that conjugating $\olA$ by an element of $GL(n, q)$ will change the matrices in $\olA$, but not the polynomials in the null ideal $N_{\F_q}(\olA)$. Moreover, up to conjugacy by an element in $GL(n, q)$, we may assume that $W$ has basis $e_1, e_2, \ldots, e_m$, where $e_i$ is the standard basis vector with $1$ in the $i^\text{th}$ component and 0 elsewhere. Under this basis, the matrices in $\olA$ are block matrices of the form 
$\big(\begin{smallmatrix}
A_1 & A_2 \\ 0 & A_3
\end{smallmatrix}\big)$ 
where $A_1$ is $m \times m$ and $A_3$ is $(n-m) \times (n-m)$. 

One consequence of this representation is that every matrix in $S(W)$ has a reducible characteristic polynomial. As shown in the proof of \cite[Thm. 3]{BrawCarLev}, $\phi_{q, n}$ is the least common multiple of all monic polynomials in $\F_q[X]$ of degree $n$. Hence, $\phi_{q, n} \in N_{\F_q}(\overline{A})$, because the characteristic polynomial of each matrix in $\olA$ divides $\phi_{q,n}$. However, if $\phi$ is the quotient of $\phi_{q,n}$ by an irreducible polynomial in $\F_q[X]$ of degree $n$, then $\phi \in N_{\F_q}(\olA)$, but $\phi \notin N_{\F_q}(M_n(\F_q))$. Thus, $N_{\F_q}(\olA)$ properly contains $N_{\F_q}(M_n(\F_q))$.

Now, assume that $\olA$ is of Type II of Proposition \ref{Racine classification}, so that $\olA = C_{{\rm End}_{\F_q}(V)}(\F_{q^l})$ for some prime $l$ dividing $n$. Then, by \cite[Thm. VIII.10]{McD}, we have $\olA \cong M_{n/l}(\F_{q^l})$, and so 
\begin{equation*}
N_{\F_q}(\olA)=(\phi_{q^l,n/l}(X))\supsetneq (\phi_{q,n}(X)) = N_{\F_q}(M_n(\F_q)).
\end{equation*}
As before, the null ideal of $\olA$ strictly contains the null ideal of $M_n(\F_q)$.
\end{proof}

\begin{Thm}\label{Uniqueness of M_n(D)}
Let $D$ be a Dedekind domain with finite residue fields. Let $A$ be a $D$-algebra of finite type with standard assumptions. Assume that $n \geq 1$ is such that $A$ can be embedded in $M_n(D)$. Then, $\Int_K(A) = \Int_K(M_n(D))$ if and only if $A \cong M_n(D)$.
\end{Thm}
\begin{proof}
Clearly, $A \cong M_n(D)$ implies that $\Int_K(A) = \Int_K(M_n(D))$. So, assume that $\Int_K(M_n(D)) = \Int_K(A)$. As we will prove shortly in Lemma \ref{wellbehaviourlocalization}, $\Int_K(A)$ (and likewise $\Int_K(M_n(D))$) is well-behaved with respect to localization at primes of $D$: for each prime $P$ of $D$, we have $\Int_K(A)_P = \Int_K(A_P)$. Hence, $\Int_K(M_n(D_P)) = \Int_K(A_P)$ for each $P$. Since $D$ is Dedekind, $D_P$ is a discrete valuation ring, so there exists $\pi \in D_P$ such that $PD_P = \pi D_P$. Moreover, we have $D_P/\pi D_P \cong D/P$ and $A_P / \pi A_P \cong A/PA$, so that $N_{D_P/\pi D_P}(A_P / \pi A_P) = N_{D/P}(A/PA)$ (and likewise for $M_n(D)$). By Lemma \ref{Null ideal lemma} (2), we conclude that the null ideals $N_{D/P}(M_n(D/P))$ and $N_{D/P}(A/P A)$ are equal for all maximal ideals $P$ of $D$.

Now, suppose by way of contradiction that the image of $A$ in $M_n(D)$ does not equal $M_n(D)$. As in Lemma \ref{FqsubalgebrasMnFq}, we may assume that the image of $A$ in $M_n(D)$ is a maximal $D$-subalgebra of $M_n(D)$. By Proposition \ref{Racine classification}, there exists a maximal ideal $P$ of $D$ such that $A/P A$ is isomorphic to a maximal subalgebra of $M_n(D/P)$. By Lemma \ref{FqsubalgebrasMnFq}, the null ideals $N_{D/P}(A/P A)$ and $N_{D/P}(M_n(D/P))$ are not equal. This is a contradiction. Therefore, $A \cong M_n(D)$.
\end{proof}

\section{General Case}\label{General case section}

We return now to the study of when $\Int_K(A)$ is nontrivial. Because of Theorem \ref{Nontriv fin gen criterion}, $A$ being an integral $D$-algebra of bounded degree can be sufficient for $\Int_K(A)$ to be nontrivial, but it is not necessary. There exist $D$-algebras $A$ that are neither finitely generated, nor algebraic over $D$ (let alone integral or of bounded degree), but for which $\Int_K(A)$ is nontrivial, as the next example shows.

\begin{Ex}\label{integralalgebraboundeddegree not necessary}
Let $D = \Z$ and let $A = \prod_{i \in \N} \Z$ be an infinite direct product of copies of $\Z$. Then, the element $(1, 2, 3, \ldots)$ cannot be killed by any polynomial in $\Z[X]$, so $A$ is not algebraic over $\Z$. However, since operations in $A$ are done component-wise, any polynomial in $\Int(\Z)$ is also in $\Int_\Q(A)$. Hence, $\Int_\Q(A) = \Int(\Z)$, so in particular $\Int_\Q(A)$ is nontrivial. 
\end{Ex}

Ultimately, the previous example works because for each prime $p$ there exists a polynomial that sends each element of $A/pA$ to 0. More explicitly, each element of $\prod_{i \in \N} \F_p$ is killed by the polynomial $X^p-X$. This suggests that for $\Int_K(A)$ to be nontrivial, it may be enough that there exists a finite index prime $P$ of $D$ with $A/PA$ algebraic of bounded degree over $D/P$ (since $D/P$ is a field in this case, this is equivalent to having $A/PA$ be integral of bounded degree over $D/P$). We will prove below in Theorem \ref{criterion} that if $D$ is a Dedekind domain, then this condition is necessary and sufficient for $\Int_K(A)$ to be nontrivial.

Our work will involve localizing $D$, $A$, and $\Int_K(A)$ at $P$, and exploiting properties of $D_P$. In \cite[Prop. 3.2]{Wer}, it is shown that if $D$ is Noetherian and $A$ is a free $D$-module of finite rank, then $\Int_K(A)_P = \Int_K(A_P)$ (in fact, \cite[Prop. 3.2]{Wer} will hold if $A$ is merely finitely generated as a $D$-module). The next lemma shows that we can drop this finiteness assumption if $D$ is Dedekind.

\begin{Lem}\label{wellbehaviourlocalization}
Let $D$ be a Dedekind domain and $A$ a $D$-algebra with standard assumptions. Let $P$ be a prime ideal of $D$. Then $\Int_K(A_P) = \Int_K(A)_P$.
\end{Lem}
\begin{proof}
The containment $\Int_K(A)_P \subseteq \Int_K(A_P)$ follows from the proof of \cite[Prop. 3.2]{Wer}, which itself is an adaptation of a technique of Rush involving induction on the degrees of the relevant polynomials (see \cite[Thm. I.2.1]{CaCh} or \cite[Prop. 1.4]{Rush}). 

For the other inclusion, let $f \in \Int_K(A_P)$ and write $f(X)=\frac{g(X)}{d}$ for some $g \in D[X]$ and $d\in D \setminus \{0\}$. Since $D$ is Dedekind, we may write $dD=P^aI$, where $a \geq 0$ and $I$ is an ideal of $D$ coprime with $P$ (possibly equal to $D$ itself). If $a=0$ then $f \in D_P[X] \subseteq \Int_K(A)_P$. If $a \geq 1$, let $c\in I \setminus P$. We claim that $cf \in\Int_K(A)$, from which the statement follows since $c \in D \setminus P$. 

If $Q \subset D$ is a prime ideal different from $P$, then $cf \in D_Q[X] \subseteq \Int_K(A_Q)$; that is, $cf(A_Q) \subset A_Q$. Now, $f(A) \subseteq f(A_P) \subseteq A_P$ by assumption, so $cf(A) \subset cA_P = A_P$, since $c \notin P$. Since $A=\bigcap_{Q\in{\rm Spec}(D)} A_Q$, it follows that $cf(A)\subset A$, and we are done.
\end{proof}


Recall (Definition \ref{Null ideal}) that the null ideal of $A$ in $R$ is $N_R(A) = \{ f \in R[X] \mid f(A) = 0\}$.

\begin{Prop}\label{equivalent conditions}
Let $D$ be a Dedekind domain and $A$ a $D$-algebra with standard assumptions. Let $P$ be a prime ideal of $D$. Then, the following are equivalent.
\begin{enumerate}[(1)]
\item $N_{D/P}(A/PA) \supsetneq (0)$.
\item $D_P[X] \subsetneq \Int_K(A_P)$.
\item $D/P$ is finite and $A/PA$ is a $D/P$-algebraic algebra of bounded degree.
\end{enumerate}
\end{Prop}
\begin{proof}
$(1) \Rightarrow (2)$ Let $g\in D[X]$ be a monic pullback of a nontrivial element $\overline{g}\in N_{D/P}(A/PA)$ and let $\pi\in P\setminus P^2$. Then, $g(A_P)\subseteq PA_P=\pi A_P$, so $\frac{g(X)}{\pi}\in \Int_K(A_P)\setminus D_P[X]$. 

$(2) \Rightarrow (1)$ Let $f(X)=\frac{g(X)}{d} \in \Int_K(A_P) \setminus D_P[X]$, with $g \in D[X] \setminus P[X]$ and  $d\in P$. Let $v_P$ denote the canonical valuation on $D_P$. If $v_P(d)=e>1$ and $\pi \in P\setminus P^2$, then $\pi^{e-1} f(X)$ is still an element of $\Int_K(A_P)$ which is not in $D_P[X]$. So, $g(A_P) \subseteq \frac{d}{\pi^{e-1}} A_P \subseteq \pi A_P$. Hence, $\overline{g}\in(D_P/PD_P)[X]\cong(D/P)[X]$ is a nontrivial element of $N_{D/P}(A/PA)$.

$(1) \Leftrightarrow (3)$ Note that
\begin{equation*}
N_{D/P}(A/PA)=\bigcap_{\overline{a}\in A/PA}N_{D/P}(\overline{a})=\bigcap_{\overline{a}\in A/PA}(\mu_{\overline{a}}(X))
\end{equation*}
where, for each $\overline{a}\in A/PA$, $\mu_{\overline{a}}\in (D/P)[X]$ is the minimal polynomial of $\overline{a}$ over the field $D/P$.

If $N_{D/P}(A/PA)$ is nonzero, then it is equal to a principal ideal generated by a monic non-constant  polynomial $\overline{g}\in (D/P)[X]$. Since $N_{D/P}(A/PA)\subseteq N_{D/P}(D/P)$, it follows that $D/P$ is finite (if not, then $N_{D/P}(D/P) = (0)$, because the only polynomial which is identically zero on an infinite field is the zero polynomial). Moreover, each element $\overline{a}\in A/PA$ is algebraic over $D/P$ (otherwise the corresponding $N_{D/P}(\overline{a})$ is zero) and its degree over $D/P$ is bounded by $\deg(\overline{g})$. 

Conversely, assume $D/P$ is finite and $A/PA$ is a $D/P$-algebraic algebra of bounded degree $n$. Then,  there are finitely many polynomials over $D/P$ of degree at most $n$, and the product of all such polynomials is a nontrivial element of $N_{D/P}(A/PA)$.
\end{proof}

We can now establish the promised criterion for $\Int_K(A)$ to be nontrivial.

\begin{Thm}\label{criterion}
Let $D$ be a Dedekind domain and let $A$ be a $D$-algebra with standard assumptions. Then $\Int_K(A)$ is nontrivial if and only if there exists a prime ideal $P$ of $D$ of finite index such that $A/PA$ is a $D/P$-algebraic algebra of bounded degree.
\end{Thm}
\begin{proof}
Clearly, $D[X]\subsetneq \Int_K(A)$ if and only if there exists a prime ideal $P\subset D$ such that the two $D$-modules $D[X]$ and $\Int_K(A)$ are not equal locally at $P$, that is, $D_P[X]\subsetneq \Int_K(A)_P$. Since $\Int_K(A)_P=\Int_K(A_P)$ by Lemma \ref{wellbehaviourlocalization}, we can apply Proposition \ref{equivalent conditions} and we are done.
\end{proof}

\begin{Ex}\label{Nontriv examples}
Theorem \ref{criterion} applies to the following examples.
\begin{itemize}
\item[(1)] Let $D=\Z$ and $A=\overline{\Z}$, the absolute integral closure of $\Z$. Then, for each $n\in\N$, there exists $\alpha\in\overline{\Z}$ of degree $d>n$ such that $O_{\Q(\alpha)}=\Z[\alpha]$. It follows that for each prime $p\in\Z$, $\overline{\Z}/p\overline{\Z}$ is an algebraic $\Z/p\Z$-algebra of unbounded degree. Thus, $\Int_{\Q}(\overline{\Z})=\Z[X]$.

\item[(2)] Let $D=\Z_{(p)}$ and $A=\Z_p$. Then, $\Z_p/p\Z_p\cong\Z/p\Z$, so $\Z_{(p)}[X]\subsetneq \Int_{\Q}(\Z_p)$.

\item[(3)] Let $D=\Z$ and $A=\widehat{\Z}=\prod_{p\in\mathbb{P}}\Z_p$, the profinite completion of $\Z$, where $\mathbb{P}$ denotes the set of all prime numbers. For each prime $p\in\Z$, we have $p\widehat{\Z}=\prod_{p'\not=p}\Z_{p'}\times p\Z_p$, so $\widehat{\Z}/p\widehat{\Z}\cong \Z_p/p\Z_p\cong\Z/p\Z$. Thus, $\Z[X]\subsetneq\Int_{\Q}(\widehat{\Z})$.
\end{itemize}
\end{Ex}

If $\whA$ is the $P$-adic completion of a $D$-algebra $A$, then we can say more about $\Int_K(\whA)$. The following lemma also appears in \cite{IntdecompII}. We include it in its entirety since the proof is quite short.

\begin{Lem}\label{DVR lemma}
Let $D$ be a discrete valuation ring (DVR) with maximal ideal $P = \pi D$. Let $A$ be a $D$-algebra with standard assumptions, and let $\whA$ be the $P$-adic completion of $A$. Then, $\Int_K(\whA) = \Int_K(A)$.
\end{Lem}
\begin{proof}
The containment $\Int_K(\whA) \subseteq \Int_K(A)$ is clear, since $A$ embeds in $\whA$. Conversely, let $f \in \Int_K(A)$ and $\alpha \in \whA$. Suppose $f(X) = g(X)/\pi^k$, where $g \in D[X]$ and $k \in \N$. If $k = 0$, then the conclusion is clear, so assume that $k > 1$. 

Via the canonical projection $\whA \to A/\pi^k A$, we see that there exists $a\in A$ such that $\alpha \equiv a \pmod{\pi^k \whA}$. Since the coefficients of $g$ are central in $A$, we get $g(\alpha)\equiv g(a) \pmod{\pi^k \whA}$. Thus, $f(\alpha)=f(a)+\lambda/\pi^k$, where $\lambda\in\pi^k \whA$, so that $f(\alpha)\in \whA$. Hence, $f \in \Int_K(\whA)$ and $\Int_K(\whA) = \Int_K(A)$.
\end{proof} 

Thus, in Example \ref{Nontriv examples} (2), we have $\Int_\Q(A) = \Int(\Z_{(p)})$. Moreover, in Example \ref{Nontriv examples} (3) we have $\Int_\Q(A) = \Int(\Z)$ (see also \cite{ChabPer} where the profinite completion of $\Z$ was considered in order to study the polynomial overrings of $\Int(\Z)$). A more general example, which results in proper containments among all of $D[X]$, $\Int_K(A)$, and $\Int(D)$, is the following.

\begin{Ex}\label{DVR example}
Let $D$ be a DVR with maximal ideal $P = \pi D$ and finite residue field. Let $A$ be a $D$-algebra of finite type with standard assumptions and such that $\Int_K(A) \subsetneq \Int(D)$. Let $\whA$ be the $P$-adic completion of $A$. Then, $P$ satisfies the conditions of Theorem \ref{criterion} with respect to $A$, so $D[X] \subsetneq \Int_K(A)$; and $\Int_K(\whA) = \Int_K(A)$ by Lemma \ref{DVR lemma}. Thus,
\begin{equation*}
D[X] \subsetneq \Int_K(\whA) = \Int_K(A) \subsetneq \Int(D).
\end{equation*}
In general, $\whA$ is not finitely generated as a $D$-module (this is the case, for instance, when $A$ is countable but $\whA$ is uncountable). So, $\whA$ can provide an example of a $D$-algebra that is not finitely generated and for which the integer-valued polynomial ring is properly contained between $D[X]$ and $\Int(D)$.
\end{Ex}

\begin{Rem}\label{Quaternion example again}
Lemma \ref{DVR lemma} also gives us another approach to Example \ref{Quaternion example}. With notation as in that example, we have $\whA \cong M_2(\Z_p)$ (indeed, this follows from the fact that $A/p^k A \cong M_2(\Z/p^k \Z)$ for all $k > 0$). Thus, $\Int_\Q(A) = \Int_\Q(M_2(\Z_p)) = \Int_\Q(M_2(\Z_{(p)}))$ even though $A \not\cong M_2(\Z_{(p)})$.
\end{Rem}

We close this paper by using the conditions of Proposition \ref{equivalent conditions} to prove that when $D$ is Dedekind, $\Int_K(A)$ has Krull dimension 2. This result was shown by Frisch \cite[Thm. 5.4]{Fri1} in the case where $A$ is of finite type. Our work does not require $A$ to be finitely generated, and somewhat surprisingly does not require a full classification of the prime ideals of $\Int_K(A)$.

Recall that a nonzero prime ideal $\mfP$ of $\Int_K(A)$ is called unitary if $\mfP \cap D \ne (0)$, and is called non-unitary if $\mfP \cap D = (0)$.

\begin{Thm}\label{Height thm}
Let $D$ be a Dedekind domain and let $A$ be a $D$-algebra with standard assumptions. Let $\mfP$ be a nonzero prime ideal of $\Int_K(A)$.
\begin{enumerate}[(1)]
\item If $\mfP$ is non-unitary, then $\mfP$ has height 1.
\item If $\mfP$ is unitary, then let $P = \mfP \cap D$.
\begin{itemize}
\item[(i)] If $P$ does not satisfy any of the conditions of Proposition \ref{equivalent conditions}, then $\mfP$ has height 2.

\item[(ii)] If $P$ satisfies one of the conditions of Proposition \ref{equivalent conditions}, then $\mfP$ is maximal 
and has height at most 2.
\end{itemize}
\end{enumerate}
\end{Thm}
\begin{proof}
(1) Following \cite[Lem. 5.3]{Fri1}, the non-unitary prime ideals of $\Int_K(A)$ are in one-to-one correspondence with the prime ideals of $K[X]$. Since $K[X]$ has dimension 1, the non-unitary primes of $\Int_K(A)$ are all of height 1.

(2) Let $P$ be a nonzero prime of $D$. Assume first that $P$ does not satisfy any of the conditions of Proposition \ref{equivalent conditions}. Then, $D_P[X] = \Int_K(A_P) = \Int_K(A)_P$. It follows that the unitary primes of $\Int_K(A)$ are in one-to-one correspondence with the primes of $D_P[X]$. Since $D$ is Dedekind, we know that $D_P[X]$ has dimension 2; hence, all the primes of $\Int_K(A)$ under consideration have height 2.

For the remainder of the proof, assume that $P = \mfP \cap D$ does satisfy the conditions of Proposition \ref{equivalent conditions}. Since $\mfP \cap D = P$, the prime ideal $\mfP$ survives in $\Int_K(A)_P=\Int_K(A_P)$ and clearly its extension $\mfP^e$ is still a prime unitary ideal (so, $\mfP^e\cap D_P=PD_P$). It is sufficient to show that $\mfP^e$ is a maximal ideal, so we may work over the localizations. Thus, without loss of generality we will assume that $D$ is a DVR. In particular, this means that $P =\pi D$, for some $\pi\in D$.

Let $\overline{g}\in N_{D/P}(A/PA)$, $\overline{g}\not=0$, and let $g\in D[X]$ be a pullback of $g(X)$. Then $g(A) \subseteq PA=\pi A$. Consequently, for each $f\in\Int_K(A)$ we have $(g\circ f)(A)\subset \pi A$. Consider the ideal $\mfA = \{F \in \Int_K(A) \mid F(A) \subseteq PA\}$ of $\Int_K(A)$. Because $P = \pi  D$ is principal, we have $\mfA = \pi \Int_K(A)$, which is contained in $\mfP$. Hence, for each $f \in \Int_K(A)$, $g \circ f \in \mfP$. 

Now, if we consider the $D/P$-algebra $\Int_K(A)/\mfP$, we see that each element of $\Int_K(A)/\mfP$ is annihilated by $\overline{g}(X)$. But $\Int_K(A)/\mfP$ is a domain, and for it to be annihilated by a nonzero polynomial, it must be finite. Thus, in fact $\Int_K(A)/\mfP$ is a finite field, and so $\mfP$ is maximal.

Finally, to show that $\mfP$ has height at most 2, let $\mfQ$ be a prime of $\Int_K(A)$ such that $(0) \subsetneq \mfQ \subseteq \mfP$. If $\mfQ$ is unitary, then we have $\mfQ \cap D = P$, and by our work above $\mfQ$ is maximal, hence equal to $\mfP$. If $\mfQ$ is non-unitary, then it has height 1 by part (1) of the theorem. It follows that $\mfP$ has height at most 2.
\end{proof}

\begin{Cor}\label{Krull dimension}
Let $D$ be a Dedekind domain with quotient field $K$. Let $A$ be a $D$-algebra with standard assumptions. Then, $\Int_K(A)$ has Krull dimension 2.
\end{Cor}
\begin{proof}
If $\Int_K(A) = D[X]$, then its dimension equals that of $D[X]$, which is 2. So, assume that $\Int_K(A)$ is nontrivial. By Theorem \ref{criterion}, there exists a prime $P$ of $D$ that satisfies the conditions of Proposition \ref{equivalent conditions}.

Let $\mfP = \{f \in \Int_K(A) \mid f(0) \in P \}$. Since $\Int_K(A) \subseteq \Int(D)$, $\mfP$ is an ideal of $\Int_K(A)$, and it is easily seen to be prime and unitary, with $\mfP \cap D = P$. Moreover, it contains the non-unitary ideal $XK[X] \cap \Int_K(A)$. Hence, $\mfP$ has height at least 2, and so $\dim(\Int_K(A)) \geq 2$. However, $\dim(\Int_K(A)) \leq 2$ by Theorem \ref{Height thm}, so we conclude that $\dim(\Int_K(A)) = 2$.
\end{proof}

\subsection*{Acknowledgments}
\noindent This research has been supported by  the grant ``Assegni Senior'' of the University of Padova. The authors wish to thank the referee for several suggestions which improved the quality of the paper.


\end{document}